\documentclass{article}
\usepackage{pinlabel}
\usepackage{color}
\usepackage{graphicx}
\usepackage{amsmath}
\usepackage{amssymb}
\usepackage{amsthm}

\newcommand{\pic}[2]{\raisebox{-.5\height}{\includegraphics[scale=#2]{#1}}}

\def\smoothing{\pic{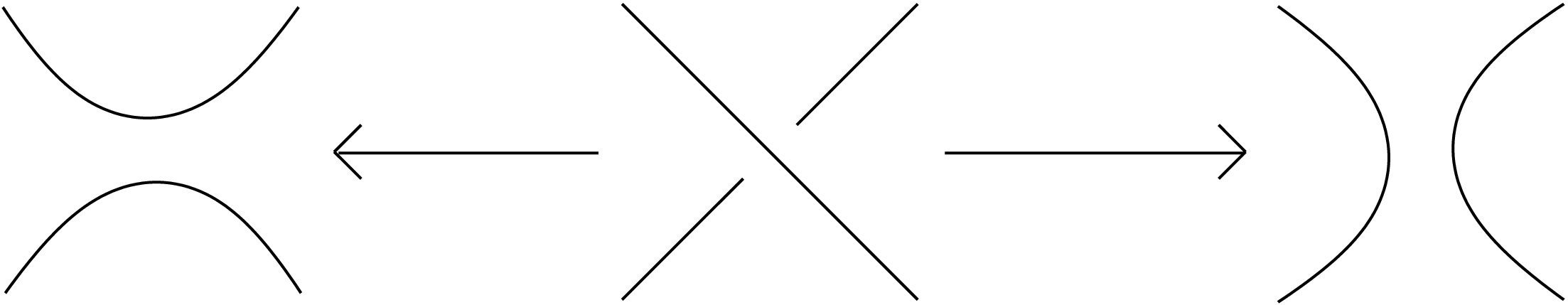}{.250}}
\def\noextraedges{\pic{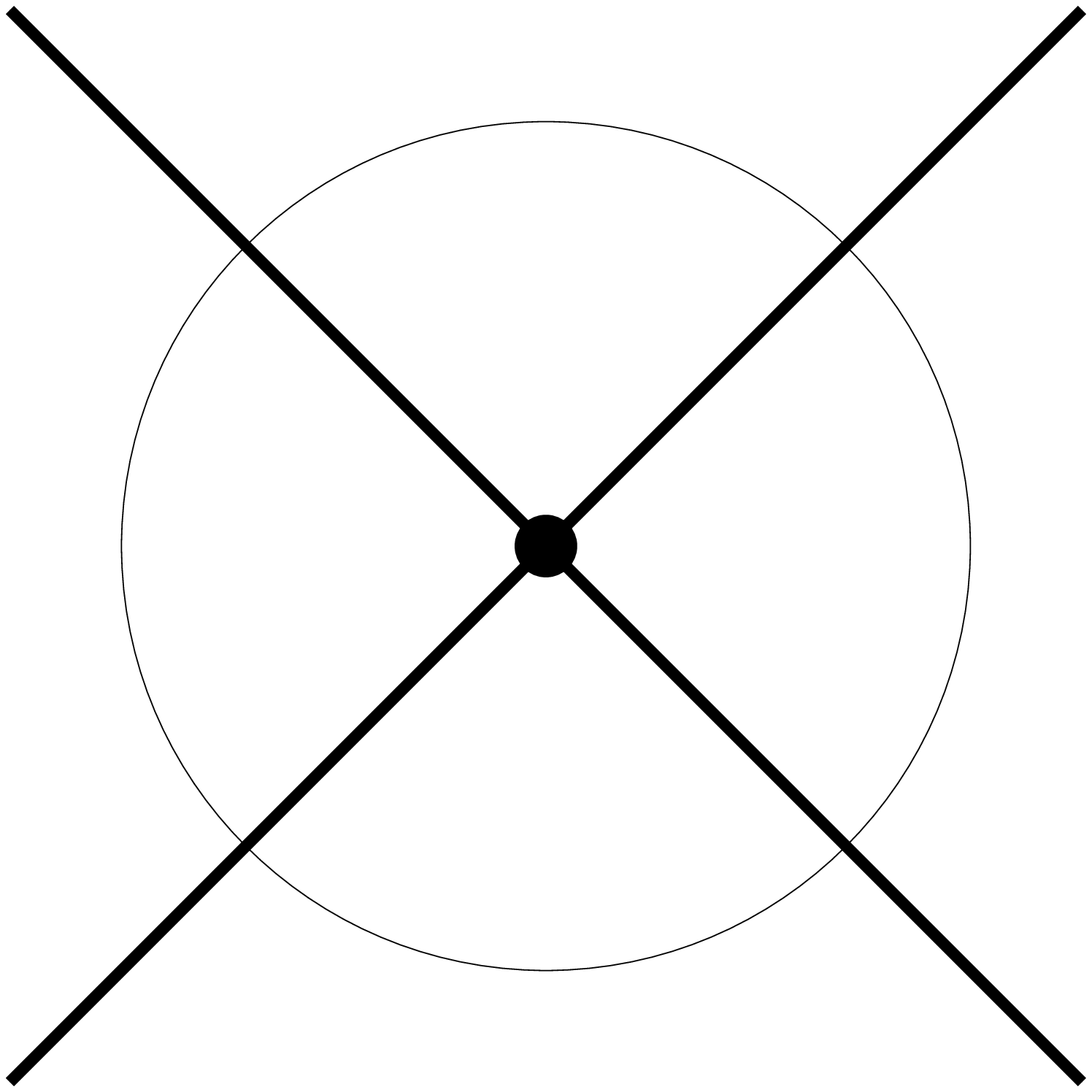}{.150}}
\def\extraedgesone{\pic{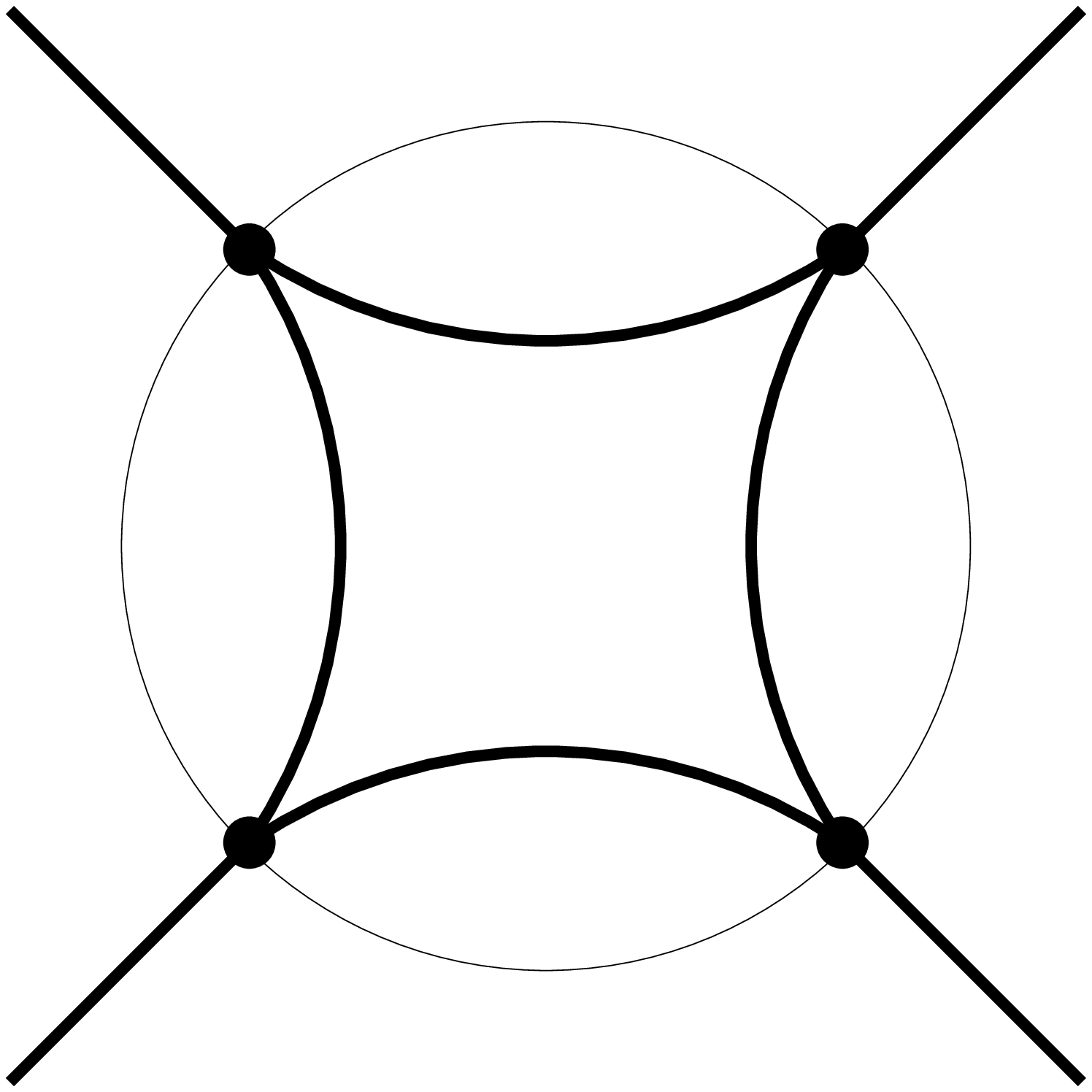}{.150}}
\def\pretzel{\pic{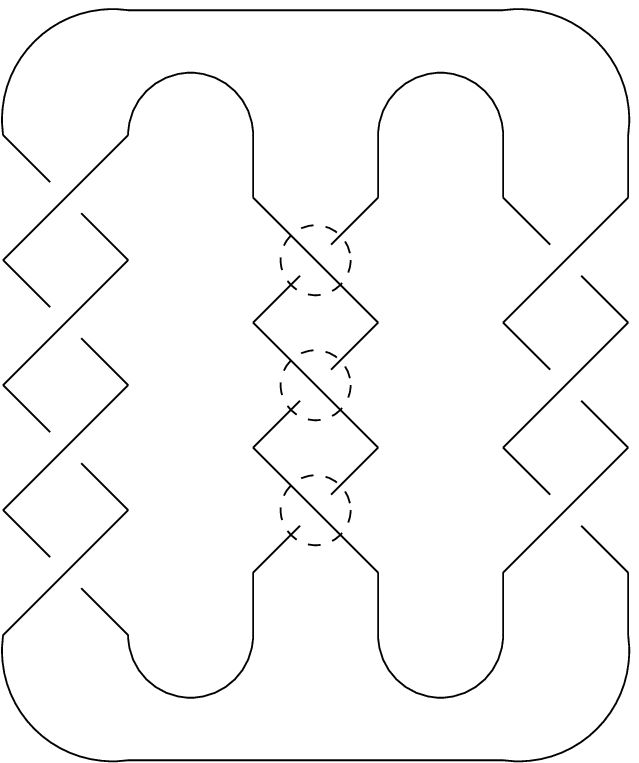}{.400}}
\def\Apretzel{\pic{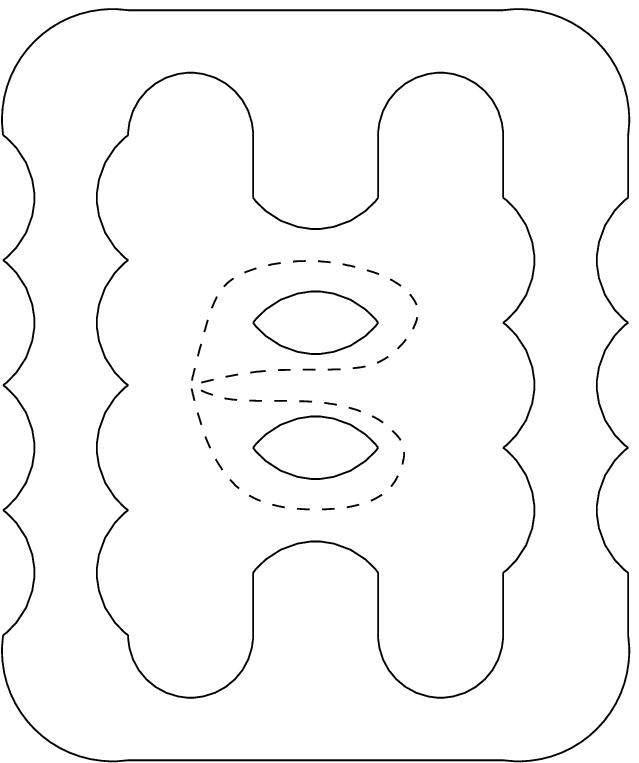}{.400}}
\def\Bpretzel{\pic{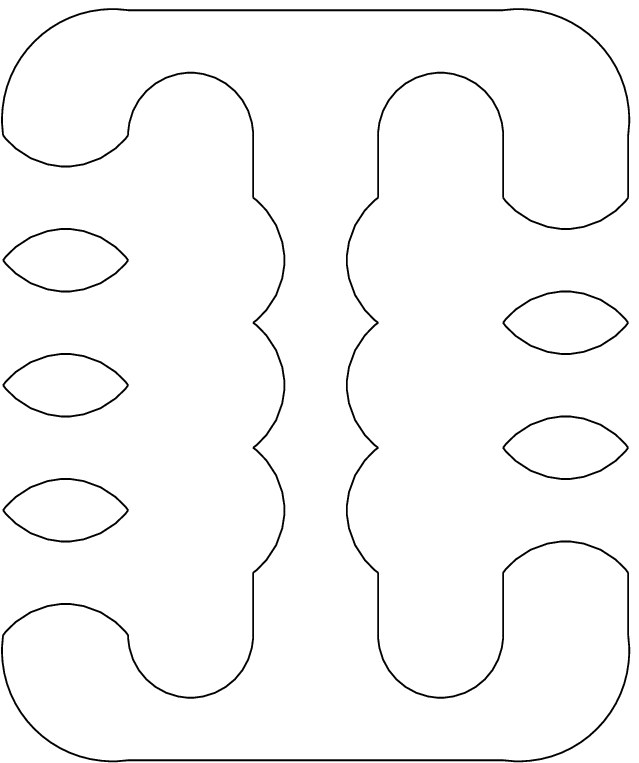}{.400}}

\newcommand{\bc}{\begin{center}}
\newcommand{\ec}{\end{center}}

\newcommand{\be}{\begin{equation}}
\newcommand{\ee}{\end{equation}}

\newcommand{\beqn}{\begin{eqnarray*}}
\newcommand{\eeqn}{\end{eqnarray*}}

\newtheorem{theorem}{Theorem}
\newtheorem{corollary}[theorem]{Corollary}
\newtheorem{lemma}[theorem]{Lemma}

\newtheorem{remark}{Remark}

{\par\smallskip}
\renewenvironment{proof}[1][Proof]{\textit{#1.} }{\hfill \rule{0.5em}{0.5em}}

\begin{document}

\title{A geometric characterization of the upper bound for the span of the Jones polynomial}
\author{J. Gonz\'alez-Meneses and P. M. G. Manch\'on}
\maketitle
%\date{July 17, 2009}  %La fecha ya sale sola

\begin{abstract}
Let $D$ be a link diagram with $n$ crossings, $s_A$ and $s_B$ its extreme states and $|s_AD|$ (resp. $|s_BD|$) the number of simple closed curves that appear when smoothing $D$ according to $s_A$ (resp. $s_B$). We give a general formula for the sum $|s_AD|+|s_BD|$ for a $k$-almost alternating diagram $D$, for any $k$, characterizing this sum as the number of faces in an appropriate triangulation of an appropriate surface with boundary. When $D$ is dealternator connected, the triangulation is especially simple, yielding $|s_AD|+|s_BD|=n+2-2k$. This gives a simple geometric proof of the upper bound of the span of the Jones polynomial for dealternator connected diagrams, a result first obtained by Zhu \cite{Zhu}. Another upper bound of the span of the Jones polynomial for dealternator connected and dealternator reduced diagrams, discovered historically first by Adams et al \cite{Adams et al}, is obtained as a corollary. As a new application, we prove that the Turaev genus is equal to the number $k$ of dealternator crossings for any dealternator connected diagram.
\end{abstract}

\vspace{0.3cm}

\noindent {\bf Keywords} $k$-almost alternating diagram, circle number, surgery, dealternator connected diagram, dealternator reduced diagram, Jones polynomial, span.

\section{Introduction}

Every link diagram $D$ has two related families of circles, $s_AD$ and $s_BD$, obtained from $D$ by applying, respectively, $A$-smoothing or $B$-smoothing to each of its crossings, as in Figure~\ref{Fsmoothing}. We denote by $|s_AD|$ (respectively $|s_BD|$), the number of circles in $s_AD$ (resp. $s_BD$). There is a well known upper bound~\cite{Lickorish} for the span of the Kauffman bracket $\langle D \rangle$, that is, the difference between the extreme degrees of $\langle D\rangle$, for a link diagram $D$ with $n$ crossings:
$$
{\rm span}(\langle D \rangle) \leq 2n+2(|s_AD|+|s_BD|)-4.
$$
\begin{figure}[ht!]
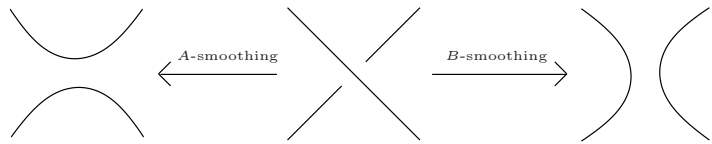

\labellist
\small
\pinlabel {\tiny $A$-smoothing} at 330 130
\pinlabel {\tiny $B$-smoothing} at 735 130
\endlabellist
  \begin{center}
\smoothing
  \end{center}
\caption{$A$ and $B$-smoothing of a crossing}\label{Fsmoothing}
\end{figure}

We will refer to the sum $|s_AD|+|s_BD|$ as the {\it circle number} of the diagram~$D$. The aim of this paper is to provide a general formula of the circle number, characterizing it as the number of faces of an appropriate triangulation of an appropriate surface.

\vspace{0.3cm}

For a connected diagram $D$, its projection yields a triangulation of $S^2$, where we allow the faces to be polygons with at least one edge, not just triangles. As every vertex in this triangulation has valence 4, its faces can be coloured white and black giving a chessboard colouring, which means that any edge is always boundary of both colours. When $D$ is alternating, the components of $s_AD$ are the (straightened) boundaries of the (let say) white faces, and the components of $s_BD$ are the boundaries of the black faces, hence $|s_AD|+|s_BD|$ is the total number of faces in the triangulation. We can count the Euler characteristic of the sphere $S^2$:
$$
n-2n+(|s_AD|+|s_BD|)=2,
$$
therefore $|s_AD|+|s_BD|=n+2$ for any connected alternating diagram with $n$ crossings.

\vspace{0.3cm}

In particular, when $D$ is connected and alternating, ${\rm span}(\langle D \rangle) \leq 4n$. If $D$ is in addition reduced, then we have the equality ${\rm span}(\langle D \rangle) =4n$ (see \cite{Lickorish}).

\vspace{0.3cm}

What happens for non-alternating diagrams? To understand the answer, we look at the concept of $k$-almost alternating diagram, a notion introduced by Adams~\cite{AdamsBook}. A diagram $D$ is said to be {\it $k$-almost alternating} if it has a set of $k$ crossings (called dealternators), and not less than $k$, such that $D$ is alternating if we switch all these crossings. Every diagram is $k$-almost alternating for exactly one $k\geq 0$. Of course, the $0$-almost alternating diagrams are the alternating diagrams. A $1$-almost alternating diagram is just called an almost alternating diagram.

\vspace{0.3cm}

Other two definitions are required. To simplify notation, we will identify each crossing of $D$ with its corresponding point in the projection, hence every dealternator is identified with a point of $S^2$. The diagram $D$ is called {\it dealternator connected} \cite{Adams et al} if there is no simple closed curve in $S^2$ intersecting (transversely) the projection of $D$ in a nonempty set of  dealternators. Equivalently, each diagram $D_i$ ($i=1, \dots , 2^k$) obtained by smoothing all $k$ dealternators in every possible way, is connected. A diagram $D$ is called dealternator reduced \cite{Adams et al} if there is no simple closed curve in $S^2$ intersecting (transversely) the projection of $D$ in exactly one non-dealternator crossing and possibly in some dealternators. Equivalently, each diagram $D_i$ ($i=1, \dots , 2^k$) obtained by smoothing all $k$ dealternators in every possible way, is reduced.

\vspace{0.3cm}

In \cite[Theorem 4]{Zhu}, Zhu proves that ${\rm span}(\langle D \rangle) \leq 4(n-k)$ if $D$ is a dealternator connected $k$-almost alternating diagram with $n$ crossings.

\vspace{0.3cm}

In \cite[Theorem 4.4]{Adams et al}, Adams et al. proved that ${\rm span}(\langle D \rangle ) \leq 4(n-k-2)$ if $D$ is a dealternator connected and dealternator reduced $k$-almost alternating diagram ($k\geq 1$) with $n$ crossings.  Historically, this result was proved before Zhu's theorem above.

\vspace{0.3cm}

Indeed, all these results provide the corresponding upper bound for the span of the Jones polynomial $V_L(t)$ of the link $L$ represented by $D$, since $ {\rm span}(\langle D \rangle ) = 4\  {\rm span}(V_L(t))$.

\vspace{0.3cm}

The main achievement of this paper is to give a geometrical interpretation of the circle number $|s_AD|+|s_BD|$ of a $k$-almost alternating diagram, as the number of faces of an appropriate triangulation of an appropriate surface with boundary and Euler characteristic $2-3k$. This construction generalizes the situation described above for alternating diagrams, and provides nice geometric proofs of the results of Zhu and Adams.

\vspace{0.3cm}

We remark that, in \cite{Turaev}, Turaev followed a similar topological approach in order to count the circle number in the case of alternating diagrams, using a different surface, sometimes called the Turaev surface in the literature. See \cite[Section 9.4]{Peter} for a nice synthesis of his work. In \cite[Corollary 7.3]{mosca}, Dasbach et al. proved that ${\rm span}(\langle D \rangle) \leq 4(n-g)$ where $g$ is the genus of the corresponding Turaev surface. In general, $g\leq k$ for a $k$-almost alternating diagram \cite{Abe}. In the case of dealternator connected diagrams we will prove that $g=k$, in light of which the result of Zhu would also follow from \cite[Corollary 7.3]{mosca}.

\vspace{0.3cm}

The paper is organized as follows: In Section~\ref{SectionSurface} we construct a surface $S$ associated to a $k$-almost alternating diagram, and a suitable graph $\Gamma_D$ in $S$. In Section~\ref{SectionDealternatorConnected} we see that if $D$ is dealternator connected, the graph $\Gamma_D$ determines a triangulation of $S$. If $D$ has $n$ crossings, this immediately yields the equality $|s_AD|+|s_BD| = n+2-2k$ (Theorem \ref{AC-case}), obtaining a simple geometric proof of the result of Zhu. From this we deduce the result of Adams et al.~\cite{Adams et al}, using a simple argument by induction. This is done in Section~\ref{SectionDealternatorConnectedAndDealternatorReduced}. The general case is treated in Section~\ref{SectionGeneralCase}: If $S\backslash \Gamma_D$ has $r$ connected components (that we will call regions) and $s$ is the rank of its first homology group, we show that $|s_AD|+|s_BD|=r+s$ (Theorem~\ref{Theorem_rs}). We also give a formula for the circle number in terms of the number of regions ($r$), crossings ($n$) and dealternators ($k$). Namely $|s_AD| + |s_BD| = 2k + 2r - n - 2$ (Theorem~\ref{Theorem_rk}). We finish with an example of these results applied to a pretzel diagram, specifying how to draw the regions of $S\backslash \Gamma_D$ in the  plane.

\vspace{.3cm}

\noindent {\bf Acknowledgements:}  We are grateful to Hugh R. Morton for several helpful comments on a previous version of this paper.

\section{The surface and graph associated to a diagram} \label{SectionSurface}

Suppose that $D$ is a $k$-almost alternating diagram. In this section we construct a surface of Euler characteristic $2-3k$ such that $|s_AD|+|s_BD|$ is the number of faces for a suitable triangulation.

\vspace{0.3cm}

The construction of the surface $S$ is made by performing the following local surgery to $S^2$ around each dealternator of $D$. Take a small closed disc $O$ around a dealternator crossing. Let $a,b,c,d$ be the four points in which the boundary of $O$ cuts transversally the diagram~$D$, say counterclockwise. The boundary of $O$ is the union of four arcs $ab$, $bc$, $cd$ and $da$. Consider two copies of the band $[0,1]\times [0,1]$. Delete the interior of the disc $O$ and glue the two bands, identifying $\{ 0 \} \times [0,1]$ and $\{ 1 \} \times [0,1]$ of the first band with $ab$ and $dc$ respectively, and $\{ 0 \} \times [0,1]$ and $\{ 1 \} \times [0,1]$ of the second band with $bc$ and $ad$ respectively (see Figure~\ref{Fsurgery}).
\begin{figure}[ht!]
\begin{center}
 \pic{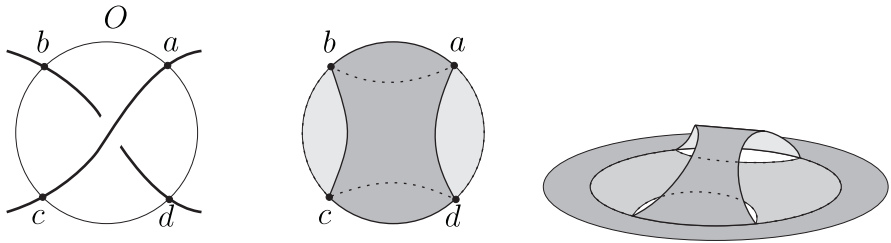}{1}
\end{center}
\caption{Local surgery around each dealternator crossing}\label{Fsurgery}
\end{figure}

What we are doing locally around each dealternator crossing is to add a hollow handle minus a disc (see Figure~\ref{Fhandle}), hence we obtain a surface $S$ which is the connected sum of $k$ torus minus the interior of $k$ discs (Figure~\ref{Fsurface}). In particular the Euler Characteristic of our surface is $2-3k$. Indeed, before deleting the interior of the discs, we have a genus $k$ handlebody, hence its Euler characteristic is $2-2k$. Deleting the $k$ discs, we get a final Euler \mbox{characteristic $2-3k$}.

\begin{figure}[ht!]
\begin{center}
 \pic{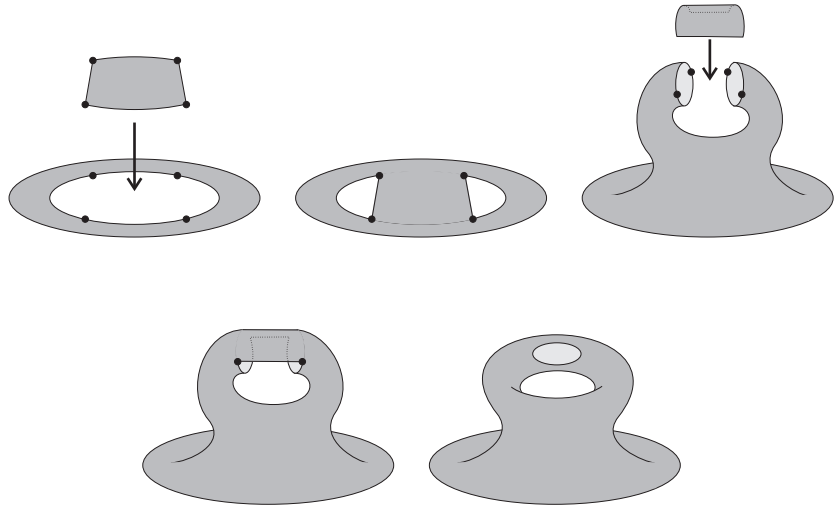}{1}
\end{center}
\caption{Local surgery adds a hollow handle minus a disc}\label{Fhandle}
\end{figure}

\begin{figure}[ht!]
\begin{center}
 \pic{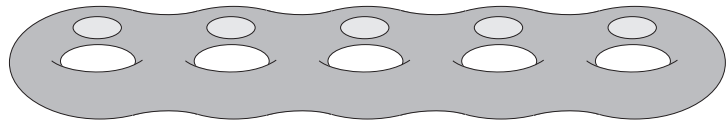}{1}
\end{center}
\caption{Type of the resulting surface, with genus $k$ and $k$ boundary components}\label{Fsurface}
\end{figure}

\vspace{0.3cm}

\begin{remark}\label{notaTuraev} In \cite{Turaev}, Turaev followed a similar topological approach in order to count the circle number in the case of alternating diagrams, using a different surface, sometimes called the Turaev surface in the literature (see also \cite{Peter}, Section 9.4).  Following \cite{mosca}, the Turaev genus of a diagram~$D$ is by definition the genus of its corresponding Turaev surface.
\end{remark}

\vspace{0.3cm}

Recall that the projection of the diagram $D$ is a graph on $S^2$, which determines a triangulation of $S^2$ admitting a chessboard colouring. In the surface $S$, we can define a similar graph, that we denote $\Gamma_D$, in the following way: Start with the sphere $S^2$ and the projection of $D$. Consider the small circle around a dealternator, along which local surgery will be applied. Recall that this circle intersects the projection of $D$ in four points, $a$, $b$, $c$ and $d$, which are interior points of their corresponding edges. Now remove the disc and glue the two bands as explained above. We complete the graph $\Gamma_D$ by considering $a$, $b$, $c$ and $d$ as vertices, and adding four edges corresponding to the segments $[0,1]\times\{i\}$ for $i=0,1$ (see Figure~\ref{Ftriangulation}). In other words, the union of the four new vertices and the four new edges is precisely the boundary component of $S$ corresponding to the given dealternator.

\begin{figure}[ht!]
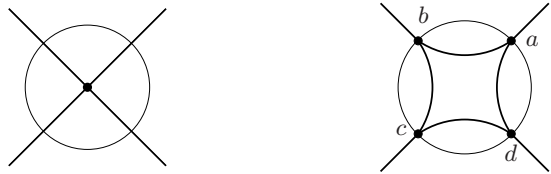

\labellist
\small
\pinlabel $a$ at 1413 340
\pinlabel $b$ at 1120 410
\pinlabel $c$ at 1060 100
\pinlabel $d$ at 1355 40
\endlabellist
  \begin{center}
\noextraedges \qquad \qquad \qquad \qquad \extraedgesone
  \end{center}
\caption{Each dealternator produces $3$ extra vertices and $4$ extra edges} \label{Ftriangulation}
\end{figure}

\vspace{.3cm}

In particular, it follows that the number of vertices in $\Gamma_D$ is $n+3k$, and the number of edges is $2n+4k$.

\vspace{.3cm}

Notice that the obtained graph $\Gamma_D$ does not yield, in general, a triangulation of $S$, since the resulting regions are not necessarily homeomorphic to a disc: a property which is equivalent to $D$ being dealternator connected. This observation will allow us to obtain a very simple proof of Zhu's result~\cite{Zhu}, as we will see in Section~\ref{SectionDealternatorConnected}. If $D$ is dealternator connected and in addition dealternator reduced, our construction will also give a simple proof of the result of Adam et al.~\cite{Adams et al}, which will be seen in Section~\ref{SectionDealternatorConnectedAndDealternatorReduced}.

\section{When $D$ is a dealternator connected diagram} \label{SectionDealternatorConnected}

If $D$ is a dealternator connected diagram, the construction of $S$ and of $\Gamma_D$ immediately determines the circle number in terms of the number of crossings and dealternators.

\begin{theorem}\label{AC-case}
If $D$ is a dealternator connected, $k$-almost alternating diagram with $n$ crossings, then
$$
|s_AD|+|s_BD|=n+2-2k.
$$
\end{theorem}

\begin{proof}
The definition of dealternator connected diagram means precisely that each region determined on the surface $S$ by $\Gamma_D$ is a disc. In other words, $\Gamma_D$ determines a triangulation of $S$, whose number of faces is precisely the circle number of $D$. Therefore, since the number of vertices is $n+3k$, the number of edges is $2n+4k$, and the Euler characteristic of $S$ is $2-3k$, we immediately obtain the formula:
$$
   (n+3k)-(2n+4k)+(|s_AD|+|s_BD|)=2-3k,
$$
from which the result follows.
\end{proof}

\vspace{.3cm}

This implies the result of Zhu mentioned in the introduction.

\begin{corollary}{\rm \cite[Theorem 4]{Zhu}}
If $D$ is a dealternator connected, $k$-almost alternating diagram with $n$ crossings, then
$$
{\rm span}(\langle D \rangle) \leq 4(n-k).
$$
\end{corollary}

\begin{proof}
It is well known~\cite{extreme} that if we denote $M= n+2|s_AD|-2$ and $m= -n-2|s_BD|+2$, then the maximal (resp. minimal) degree of the Kauffman bracket of the diagram $D$ is at most $M$ (resp. at least $m$). Hence $\mbox{span}(\langle D \rangle) \leq M-m = 2n+2(|s_AD|+|s_BD|)-4$.  As, by Theorem~\ref{AC-case}, $|s_AD|+|s_BD|=n+2-2k$ under our hypothesis, the result follows.
\end{proof}

\vspace{.3cm}

Recall that if $L$ is a link represented by a diagram $D$, the span of the Kauffman bracket $\langle D \rangle$ of $D$ is four times the span of the Jones polynomial $V_L(t)$ of $L$. Hence, if $L$ is a link represented by a dealternator connected, $k$-almost alternating diagram with $n$ crossings, then
$$
{\rm span}(V_L(t))\leq n-k.
$$

We finish this section by proving that the Turaev genus (see Remark~\ref{notaTuraev}) agrees with the dealternating number for dealternator connected diagrams. Precisely,

\begin{corollary}\label{GenusDealternatorConnected}
If $D$ is a dealternator connected $k$-almost alternating diagram, then its Turaev genus $g$ is equal to $k$.
\end{corollary}

\begin{proof}
It is well known \cite{Peter} that $2g=2+n-(|s_AD|+|s_BD|)$ where $g$ is the genus of the Turaev surface built from $D$. The result follows then from Theorem~\ref{AC-case}.
\end{proof}

\begin{remark}\label{nota} In light of Corollary~\ref{GenusDealternatorConnected}, the result of Zhu is also a consequence of~\cite[Corollary 7.3]{mosca}.
\end{remark}

\section{When $D$ is both dealternator connected and dealternator reduced} \label{SectionDealternatorConnectedAndDealternatorReduced}

The result of Adams cited in the introduction~\cite[Theorem 4.4]{Adams et al},  was originally proved by writing the Kauffman bracket of $D$ in terms of the Kauffman brackets of the connected, reduced and alternating diagrams $D_i$, $i = 1 \dots , 2^k$. Here we will deduce it from the result of Zhu, using a simple argument by induction.

\vspace{.3cm}

In order to start induction, we need a result for {\it adequate} diagrams~\cite{LT}.

\begin{theorem}\label{Theorem_adequate}{\rm \cite[Proposition 1]{LT}}
Let $D$ be an adequate diagram with $n$ crossings. Then the terms of the highest and lowest degrees in its Kauffman bracket $\langle D\rangle$ are
$$
    (-1)^{|s_AD|-1}A^{M} \qquad \mbox{and} \qquad  (-1)^{|s_BD|-1}A^{m},
$$
where $M=n+2|s_AD|-2$ and $m=-n-2|s_BD|+2$.
\end{theorem}

\vspace{.3cm}

We recall that the number $M$ (resp. $m$) above is the maximal (resp. minimal) possible degree of the Kauffman bracket $\langle D\rangle$ of an arbitrary diagram $D$. Moreover, the degree of any term in the Kauffman bracket is congruent with $m$ (and also with $M$) modulo 4 (see, for instance, \cite{extreme}). In other words, the Kauffman bracket of any diagram $D$ can be written as
\begin{equation}\label{Equation<D>}
\langle D \rangle =a_mA^m + a_{m+4}A^{m+4}+ \dots+ a_{M-4}A^{M-4} + a_MA^M,
\end{equation}
where some of the coefficients could possibly be zero. We will call $a_M$ (resp. $a_m$) the hypothetic maximal (resp. minimal) coefficient of $\langle D\rangle$.

\vspace{.3cm}
For any diagram the values of these coefficients are $a_m=(-1)^{|s_BD|-1}I(G_B^D)$ and $a_M=(-1)^{|s_AD|-1}I(G_A^D)$, where $G_B^D$ and $G_A^D$ are certain graphs, and $I(G)$ denotes certain independence number of the graph $G$ (see \cite{extreme} for details). It turns out that a diagram $D$ is adequate if and only if both graphs $G_B^D$ and $G_A^D$ are empty, which is a nice characterization of adequacy in terms of graphs --compare to \cite[Proposition 2 (ii)]{Morwen}. Since the independence number of the empty graph is one, this gives another proof of Theorem~\ref{Theorem_adequate}.

\vspace{.3cm}

In the particular case we are interested in, the hypothetic extreme coefficients of $\langle D\rangle$ can be described in terms of simpler diagrams, as follows:

\begin{lemma}\label{Lemma a_M a_m}
Let $D$ be a dealternator connected $k$-almost alternating diagram with $n$ crossings. Suppose that $k>0$ and choose in $D$ a dealternator crossing. Let $D_1$ (resp. $D_2$) be the diagram obtained by $A$-smoothing (resp. $B$-smoothing) this dealternator crossing. Then
$$
 |s_AD_1|=|s_AD|, |s_AD_2|=|s_AD|+1, |s_BD_1|=|s_BD|+1 \mbox{ and } |s_BD_2|=|s_BD|.
$$
Moreover, let $a_M$ (resp. $a_{M_1}$, $a_{M_2}$) be the hypothetic maximal coefficient of $\langle D\rangle$ (resp. $\langle D_1\rangle$, $\langle D_2\rangle$). Let $a_m$ (resp. $a_{m_1}$, $a_{m_2}$) be the hypothetic minimal coefficient of $\langle D\rangle$ (resp. $\langle D_1\rangle$, $\langle D_2\rangle$). Then
$$
    a_M=a_{M_1}+a_{M_2} \qquad \mbox{and} \qquad  a_m=a_{m_1}+a_{m_2}.
$$
\end{lemma}

\begin{proof}
The equalities $|s_AD_1|=|s_AD|$ and $|s_BD_2|=|s_BD|$ are obvious, so let us prove that $|s_AD_2|=|s_AD|+1$. Of course $|s_AD_2|=|s_AD|+\epsilon$, where $\epsilon = \pm1$, hence
$$
|s_AD_2|+|s_BD_2|=|s_AD|+|s_BD|+\epsilon .
$$
Since $D$ is a dealternator connected $k$-almost alternating diagram with $n$ crossings, by Theorem~\ref{AC-case} we know that $|s_AD|+|s_BD|=n+2-2k$.

\vspace{.3cm}

Now notice that $D_2$ is a $(k-1)$-almost alternating diagram with $n-1$ crossings. Indeed, switching the other $k-1$ dealternator crossings in $D_2$ is equivalent to first switching all the $k$ dealternator crossings of $D$ and then smoothing the selected dealternator, and any alternating diagram is still alternating after ($A$ or $B$)-smoothing any crossing. Since $D_2$ is also dealternator connected, by Theorem~\ref{AC-case} again   $|s_AD_2|+|s_BD_2|=(n-1)+2-2(k-1)$. It follows that
$$
(n-1)+2-2(k-1)=n+2-2k+\epsilon
$$
hence $\epsilon = 1$.  The equality $|s_BD_1|=|s_BD|+1$ is shown in the analogous way.

\vspace{.3cm}

In order to show that $a_M=a_{M_1}+a_{M_2}$, recall that
$$
\langle D \rangle = A \langle D_1 \rangle + A^{-1}\langle D_2 \rangle,
$$
so we need to show that $M_1=M-1$ and $M_2=M+1$. As $D_1$ and $D_2$ are diagrams with $n-1$ crossings, this is equivalent to $|s_AD_1|=|s_AD|$ and $|s_AD_2|=|s_AD|+1$, so we are done. An analogous argument gives the equality involving $a_m$.
\end{proof}

\vspace{.3cm}
We can now show in a simpler way the result by Adam et al.

\begin{corollary}{\rm \cite[Theorem 4.4]{Adams et al}}
If $D$ is a dealternator connected and dealternator reduced $k$-almost alternating diagram with $n$ crossings, and $k>0$, then
$$
{\rm span}(\langle D \rangle) \leq 4(n-k-2).
$$
\end{corollary}

\begin{proof}
By Theorem~\ref{AC-case}, we know that the hypothetical maximal value of $\mbox{span}(\langle D\rangle)$ is $M-m = 2n+2(|s_AD|+|s_BD|)-4 = 4(n-k)$. But as the Kauffman bracket has the expression~(\ref{Equation<D>}) above, it follows that the above bound will decrease by $8$ if we show that $a_M = a_m =0$.

\vspace{.3cm}

We proceed by induction on $k$. Suppose that $k=1$, that is, $D$ has only one dealternator. Denote $D_1$ (resp. $D_2$) the diagram obtained by $A$-smoothing (resp. $B$-smoothing) this dealternator crossing. By Lemma~\ref{Lemma a_M a_m}, one has $a_M=a_{M_1}+a_{M_2}$. But $D_1$ and $D_2$ are alternating, reduced diagrams, thus they are adequate~\cite{LT}. Hence Theorem~\ref{Theorem_adequate} and Lemma~\ref{Lemma a_M a_m} tell us that $a_{M_1}=(-1)^{|s_AD_1|-1}=(-1)^{|s_AD|-1}$, and on the other hand $a_{M_2}=(-1)^{|s_AD_2|-1}=(-1)^{|s_AD|}$. Therefore $a_M= a_{M_1}+a_{M_2}=0$.  The analogous argument shows that $a_m=a_{m_1}+a_{m_2}=0$, so the case $k=1$ holds.

\vspace{.3cm}

Suppose now that $k>1$ and that the result holds for diagrams with less than $k$ dealternators. Choose one dealternator of $D$ and apply $A$-smoothing (resp. $B$-smoothing)  to create the diagram $D_1$ (resp. $D_2$). Notice that both $D_1$ and $D_2$ are dealternator connected and dealternator reduced $(k-1)$-almost alternating diagrams with $n-1$ crossings. By induction hypothesis, $a_{M_1}=a_{m_1}=a_{M_2}=a_{m_2}=0$. Hence $a_M=a_{M_1}+a_{M_2}=0$ and $a_m=a_{m_1}+a_{m_2}=0$, so the result follows.
\end{proof}

\section{The general case}\label{SectionGeneralCase}

If a diagram $D$ is not dealternator connected, the graph $\Gamma_D$ does not determine a triangulation of the surface $S$, since at least one of the regions determined by $\Gamma_D$ is not homeomorphic to a disc. Nevertheless, these regions (the connected components of $S\backslash \Gamma_D$) admit a black and white colouring which extends the chessboard colouring of $S^2$: It suffices to colour the bands attached during the surgery in the natural way. Notice that, with this colouring, the components of $s_AD$ are the boundaries of the white regions, and the components of $s_BD$ are the boundaries of the black regions. Notice also that these regions have genus~$0$ (they are discs with holes), so the number of components of their boundary is determined by the rank of their first homology group (the number of holes).

\vspace{.3cm}

Therefore, the circle number $|s_AD|+|s_BD|$ is determined by the number of regions in $S\backslash \Gamma_D$, together with the ranks of their first homology groups. More precisely:

\begin{theorem}\label{Theorem_rs}
Let $D$ be a $k$-almost alternating diagram with $n$ crossings, and let $S$ and $\Gamma_D$ be defined as in Section~\ref{SectionSurface}. Let $R_1,\ldots,R_r$ be the connected components of $S\backslash \Gamma_D$, and let $s_i$ be the rank of the first homology group of $R_i$. Finally, denote $s=s_1+\cdots + s_r$.  Then
$$
   |s_AD|+|s_BD| = r+s.
$$
\end{theorem}

\begin{proof}
As we mentioned above, the circle number of $D$ is precisely the number of boundary components of all $R_i$'s. Since $R_i$ is a disc with $s_i$ holes, the number of boundary components of $R_i$ is precisely $s_i+1$, hence
$$
|s_AD|+|s_BD|=\sum_{i=1}^r(s_i+1)=r+\sum_{i=1}^r s_i= r+s.
$$
\end{proof}

\vspace{.3cm}

Theorem~\ref{Theorem_rs} gives a description of the circle number of a diagram $D$ in terms of the number $r$ of regions in $S\backslash \Gamma_D$, and the shape of each region. We can simplify this description: Using that the Euler characteristic of $S$ is $2-3k$, we can describe the circle number of $D$ in terms of $n$, $r$ and the number $k$ of dealternators.

\begin{theorem}\label{Theorem_rk}
Let $D$ be a $k$-almost alternating diagram with $n$ crossings, and let $S$ and $\Gamma_D$ be defined as in Section~\ref{SectionSurface}. Let $r$ be the number of connected components of $S\backslash \Gamma_D$. Then
$$
|s_AD| + |s_BD| = 2k + 2r - n - 2.
$$
\end{theorem}

\begin{proof}
The graph $\Gamma_D$ determines a decomposition of the surface $S$, which is not a triangulation, in general. In order to transform it into a triangulation (in which we admit polygonal faces), we just need to add some edges to the graph, in the following way.

\vspace{.3cm}

As above, suppose that $R_1,\ldots,R_r$ are the connected components of $S\backslash \Gamma_D$, $s_i$ is the rank of the first homology group of $R_i$ and $s=s_1+\cdots + s_r$. 

\vspace{.3cm}

Suppose now that a region $R_i$ is not homeomorphic to a disc, that is, $s_i~>~0$. Notice that $R_i$ is constructed from several discs (regions of the alternating diagram associated to $D$), joined by some bands. Each band corresponds to a dealternator, and contributes with four vertices to the boundary of $R_i$. Notice also that the attachment of a band can increase the rank of the first homology group of a region by at most one. Hence, there are at least $s_i$ bands involved in the construction of $R_i$.

\vspace{.3cm}

Let us show, by induction on the number $s_i$, that there are $s_i$ bands in $R_i$ such that, if we remove them, the remaining region is a single disc. Indeed, if $s_i=0$ there is nothing to show. Suppose that $s_i>0$. Notice that if we remove a band whose four vertices are in the same boundary component of $R_i$, then the resulting region is a disjoint union of ``discs with bands", and the total rank has not been modified. We know that removing all bands we get a family of discs, therefore at some point the rank must decrease, and this means that there is some band whose vertices belong to two different boundary components of $R_i$. Removing this band we decrease the rank of $R_i$ by one, and the claim follows by induction hypothesis.

\vspace{.3cm}

Consider the $s_i$ bands given by the previous claim. Two edges of each band belong to $\Gamma_D$. We now add to our graph the other two edges of each band. In this way, the region $R_i$ has been subdivided into $s_i+1$ discs, the number of vertices is the same, and the number of edges has increased by $2s_i$.  Applying the same procedure to each region, we obtain a triangulation of $S$ whose number of vertices is $n+3k$, whose number of edges is the number of edges of $\Gamma_D$ plus $2s$, that is, $2n+4k+2s$, and whose number of faces is precisely the circle number of $D$, since each region $R_i$ yields $s_i+1$ faces, which is precisely the number of circles associated to $R_i$.  Therefore, as the Euler characteristic of $S$ is $2-3k$, we obtain
$$
   (n+3k)  - (2n+4k+2s) + |s_AD|+|s_BD| = 2-3k.
$$
That is,
$$
   |s_AD|+|s_BD| = n + 2s -2k +2.
$$
Applying Theorem~\ref{Theorem_rs}, one has $2s = 2(|s_AD|+|s_BD|)-2r$, hence
$$
    |s_AD|+|s_BD| = n + 2(|s_AD|+|s_BD|)-2r -2k+2,
$$
so
$$
   |s_AD|+|s_BD|= 2k+2r-n-2,
$$
as we wanted to show.
\end{proof}

\vspace{.3cm}

We remark that Theorem~\ref{AC-case} is a corollary of this result, since in the case of a dealternator connected diagram, one has $r=|s_AD|+|s_BD|$.

\vspace{.3cm}

\noindent {\bf Example.} Figure \ref{Fcorona} exhibits the pretzel diagram $D=P(4,-3,3)$ and its corresponding families of circles $s_AD$ and $s_BD$. The diagram $D$ is a non-dealternator connected $3$-almost alternating diagram, with 10 crossings. In this case, the graph $\Gamma_D$ defined in Section~\ref{SectionSurface} does not yield a genuine triangulation of $S$, since one of the faces (regions of $S\backslash \Gamma_D$) is not a disc, as we will now see.

\begin{figure}[ht!]
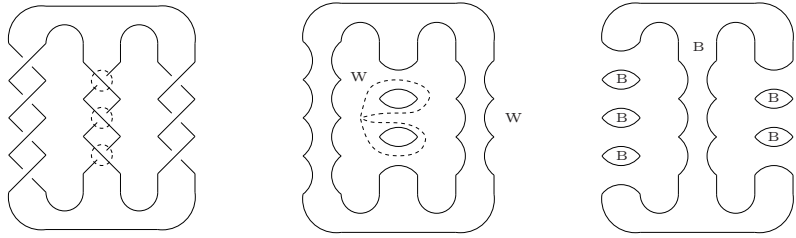

\labellist
\small
\pinlabel {\tiny \rm W} at 340 149
\pinlabel {\tiny \rm W} at 490 109

\pinlabel {\tiny \rm B} at 595 147
\pinlabel {\tiny \rm B} at 595 109
\pinlabel {\tiny \rm B} at 595 72

\pinlabel {\tiny \rm B} at 669 178

\pinlabel {\tiny \rm B} at 742 128
\pinlabel {\tiny \rm B} at 742 91

\endlabellist
  \begin{center}
\pretzel \qquad \qquad \Apretzel \qquad \qquad \Bpretzel
  \end{center}
\caption{Diagrams $D$, $s_AD$ and $s_BD$; when $A$-smoothing, a region which is not a disc emerges}\label{Fcorona}
\end{figure}

\vspace{0.3cm}

We remark that the white (resp. black) faces of $\Gamma_D$ can be drawn on the plane (actually on the sphere), as they can be obtained from the white (resp. black) regions in the chessboard colouring of the alternating diagram corresponding to~$D$, by joining some regions along the dealternators. That is, each dealternator can be seen as a pair of {\it bridges}, one of them connecting white regions, the other one black regions. 

\vspace{0.3cm}

More precisely, we can colour the faces of $S\backslash \Gamma_D$ as follows. First, choose any dealternator and consider the region of the plane that encloses the crossing point corresponding to that dealternator, in both the diagrams $s_AD$ and $s_BD$. These two regions should be considered white and black respectively. Now extend the colour white in $s_AD$ following the chessboard colouring fashion, using a neutral colour as second colour. Analogously, extend the colour black in $s_BD$ following the chessboard colouring fashion, using a neutral colour as second colour. In this way, the white and black regions of both pictures (drawn in $S^2$) are copies of the white and black regions of $S\backslash \Gamma_D$.

%Notice that this construction makes sense (there is a coherent labelling of faces into white, black or neutral), since the alternating diagram corresponding to $D$ (obtained by switching its dealternators) admits a chessboard colouring.

\vspace{.3cm}

In our example, one of the white faces is a disc with two holes, while all the other faces are discs. In the setting of Theorem~\ref{Theorem_rs}, we have $r=8$ and $s=2$, so we can check that
$$
  4+6=|s_AD|+|s_BD|= r+s=8+2 =10.
$$
And we can also check the formula of Theorem~\ref{Theorem_rk}:
$$
  4+6=|s_AD|+|s_BD| = 2k+2r-n-2= 6+16-10-2 = 10.
$$
This gives 36 as an upper bound for the span of the Kauffman bracket of $D$. According to \cite{pretzel}, we have ${\rm span}(\langle P(4,-3,3)\rangle)=28$.

\vspace{.3cm}

In this paper we have determined the value of the circle number $|s_AD|+|s_BD|$ of any diagram $D$, obtaining an exact estimation of the known upper bound $2n+2(|s_AD|+|s_BD|)-4$ of the span of the Kauffman bracket of $D$. Nevertheless, our example shows the necessity of sharpening this upper bound in some sense. From the approach given in this paper one is tempted to replace, in the above upper bound, the circle number by the number of regions $r$ of $S\backslash \Gamma_D$. This would give the better upper bound $2n+2r-4=32$ in our example. But in general this fails to be an upper bound of the span of the Kauffman bracket, as can be seen analyzing the diagram listed as $K11n151$ in the Hoste-Thistlethwaite Knot Table \cite{knotatlas}.


\begin{thebibliography}{99}

\bibitem{Abe} Abe, T. and Kishimoto, K.: The dealternating number and the alternation number of a closed 3-braid. \emph{arxiv.org/abs/0808.0573}, (2008), 1--18.

\bibitem{AdamsBook} Adams, C. C.: The Knot Book: An Elementary Introduction to the Mathematical Theory of Knots. W.H. Freeman and Company, (1994).

\bibitem{Adams et al} Adams, C. C., Brock, J. F., Bugbee, J. et al.: Almost alternating links. \emph{Topology Appl.} {\textbf 46} (1992), no. 2, 151--165.

\bibitem{BaeMorton} Bae, Y. and Morton, H.R.: The spread and extreme terms of Jones polynomials. \emph{J. Knot Theory Ramifications} {\textbf 12}, no. 3, 359--373 (2003).

\bibitem{knotatlas} Bar-Natan, D., Morrison, S. et al.: The Knot Atlas. \emph{http://katlas.org}.

\bibitem{Peter} Cromwell, P.: Knots and links. \emph{Cambridge University Press,} (2004).

\bibitem{mosca} Dasbach, O. T., Futer, D., Kalfagianni, E., Lin, X-S and Stoltzfus, N. W.: The Jones polynomial and graphs on surfaces. \emph{J. Combin. Theory Ser. B}, {\textbf 98} (2008), no. 2, 384--399.

\bibitem{Lickorish} Lickorish, W. B. R.: An introduction to Knot Theory. \emph{Graduate texts in Mathematics,} {\textbf 175}. Springer-Verlag, (1997).

\bibitem{LT} Lickorish, W. B. R. and Thistlethwaite, M.: Some links with non-trivial polynomials and their crossing-numbers. \emph{Comment. Math. Helvetici} {\textbf 63},  527--539 (1988).

\bibitem{extreme} Manch\'on, P. M. G.: Extreme coefficients of the Jones polynomial and graph theory. \emph{J. Knot Theory Ramifications} {\textbf 13}, no. 2, 277--295 (2004).

\bibitem{pretzel} Manch\'on, P. M. G.: Kauffman bracket of pretzel links.
Marie Curie Fellowships Annals, Second Volume, 118-122 (2003).\newline
{\small http://www.mariecurie.org/annals/index.html?frame3=/annals/volume2/contents.htm}

\bibitem{Morwen} Thistlethwaite, M.: On the Kauffman polynomial of an adequate link. \emph{Invent. math.}, 93, 285--296 (1988).

\bibitem{Turaev} Turaev, V. G.: A simple proof of the Murasugi and Kauffman theorems on alternating links. \emph{L'Enseignement Math.} 33, (1987), 203--225.

\bibitem{Zhu} Zhu, J.: On Kauffman brackets. \emph{J. Knot Theory  Ramifications} {\textbf 6} (1997), no 1, 125--148.

\end{thebibliography}
\end{document}